\documentclass[a4article, 10pt]{amsart}

\usepackage{enumitem}
\usepackage{amsmath}
\usepackage{amssymb, latexsym}
\usepackage{fancybox, graphicx}
\usepackage[usenames, dvipsnames]{color}
\usepackage[colorlinks=true,
        raiselinks=true,
        linkcolor=MidnightBlue,
        citecolor=Mahogany,
        urlcolor=ForestGreen,
        pdfauthor={Tobias Sutter},
        pdftitle={Isospectral Flows on a Class of Finite-Dimensional Jacobi Matrices},
        pdfkeywords={isospectral flow, zero-diagonal Jacobi matrices, block diagonal},
        pdfsubject={Technical Report},
        plainpages=false]{hyperref}
\usepackage{mathtools}
\usepackage[square, sort&compress, authoryear]{natbib}
\usepackage{lmodern}
\DeclareFontFamily{T1}{pzc}{}
\DeclareFontShape{T1}{pzc}{m}{it}{<-> [1.00] pzcmi8t}{}
\DeclareMathAlphabet{\mathpzc}{T1}{pzc}{m}{it}

\DeclareMathOperator{\tr}{tr}

\DeclareMathOperator{\sgn}{sgn}

\newcommand{\nn}{\nonumber}

\renewcommand{\leq}{\le}
\renewcommand{\geq}{\ge}
\newcommand{\Let}{\coloneqq}

\def\bmat{\left[ \begin{array}}
\def\emat{\end{array} \right]}

\newcommand{\spectrum}{{\sigma}}			
\DeclareMathOperator{\jac}{Jac}				
\DeclareMathOperator{\sym}{Sym}				
\DeclareMathOperator{\skw}{Skew}			
\newcommand{\orth}{{\mathcal{O}}}			

\newcommand{\D}{{\mathcal{D}}}				
\DeclareMathOperator{\Mj}{M_J}				
\DeclareMathOperator{\Ms}{M}				

\definecolor{mred}{rgb}{0.6, 0, 0}
\definecolor{mgreen}{rgb}{0, 0.5, 0}
\definecolor{mblue}{rgb}{0, 0, 0.5}
\definecolor{mcyan}{rgb}{0, 0.5, 0.5}

\newcommand{\R}{\ensuremath{\mathbb{R}}}
\newcommand{\N}{\ensuremath{\mathbb{N}}}
\newcommand{\Nz}{\ensuremath{\mathbb{N}_0}}

\newcommand{\lra}{\ensuremath{\longrightarrow}}

\renewcommand{\le}{\ensuremath{\leqslant}}
\renewcommand{\ge}{\ensuremath{\geqslant}}
\renewcommand{\mapsto}{\ensuremath{\longmapsto}}

\newcommand{\norm}[1]{\ensuremath{\left\lVert #1 \right\rVert}}

\newcommand{\inprod}[2]{\ensuremath{\left\langle{#1}\vphantom{\big|},\vphantom{\big|}{#2}\right\rangle}}

\newcommand{\secref}[1]{\S\ref{#1}}

\newcommand{\transp}{\ensuremath{^{\scriptscriptstyle{\top}}}}

\newcommand{\drv}{\ensuremath{\,\mathrm{d}}}

\newtheoremstyle{nonum}{4pt}{4pt}{}{}{\itshape}{.}{ }{\thmname{#1}\thmnote{ (\mdseries #3)}}
\theoremstyle{nonum}

\newtheorem{remarksnn}{Remarks and contributions}

\newtheoremstyle{nonumt}{4pt}{4pt}{\slshape}{}{\bfseries}{.}{ }{\thmname{#1}\thmnote{ (\mdseries #3)}}
\theoremstyle{nonumt}

\swapnumbers
\numberwithin{equation}{section}
\newtheoremstyle{dcstyle}{4pt}{4pt}{\itshape}{}{\bfseries}{.}{ }{}

\theoremstyle{dcstyle}
\newtheorem{theorem}[equation]{Theorem}

\newtheorem{lemma}[equation]{Lemma}
\newtheorem{proposition}[equation]{Proposition}

\theoremstyle{definition}

\theoremstyle{remark}

\newtheorem{example}[equation]{Example}

\makeatletter
\def\tagform@#1{\maketag@@@{\ignorespaces#1\unskip\@@italiccorr}}
\makeatother

\newcommand{\kreis}[1]{\unitlength1ex\begin{picture}(2.5,2.5)%
\put(0.75,0.75){\circle{2.5}}\put(0.75,0.75){\makebox(0,0){#1}}\end{picture}}

\newcommand{\Lp}[1]{\mathrm{L}_{#1}}
\newcommand{\Jacobi}{\mathcal{J}}

\allowdisplaybreaks

\title[Isospectral flows on a class of Jacobi matrices]{Isospectral flows on a class of finite-dimensional Jacobi matrices}
\author[T.\ Sutter]{Tobias Sutter}
\address[T.~Sutter, and J.~Lygeros]{Automatic Control Laboratory, ETL I28, ETH Z\"urich, Physikstrasse 3, 8092 Z\"urich, Switzerland}
\email[T.~Sutter and J.~Lygeros]{suttert@student.ethz.ch, jlygeros@control.ee.ethz.ch}
\urladdr{\url{http://control.ee.ethz.ch}}

\author[D.\ Chatterjee]{Debasish Chatterjee}
\address[D.~Chatterjee]{Systems \& Control Engineering\\ IIT-Bombay, Powai\\ 400~076 Mumbai\\ India}
\email[D.~Chatterjee]{chatterjee@sc.iitb.ac.in}
\urladdr{\url{http://www.sc.iitb.ac.in/~chatterjee}}

\author[F.\ A.\ Ramponi]{Federico A. Ramponi}
\address[F.\ A.\ Ramponi]{Universit\`a degli Studi di Brescia, Dipartimento di Ingegneria dell'Informazione \\ Via Branze 38, 25123 Brescia, Italy}
\email[F.\ A.\ Ramponi]{federico.ramponi@ing.unibs.it}

\author[J.\ Lygeros]{John Lygeros}

\thanks{This research was partially supported by the European Commission under the project MoVeS (FP7-ICT-2009.3.5), and the HYCON2 Network of Excellence (FP7-ICT-2009-5).}

\begin{document}

\begin{abstract}
We present a new matrix-valued isospectral ordinary differential equation that asymptotically block-diagonalizes $n\times n$ zero-diagonal Jacobi matrices employed as its initial condition. This o.d.e.\ features a right-hand side with a nested commutator of matrices, and structurally resembles the double-bracket o.d.e.\ studied by R.W.\ Brockett in 1991. We prove that its solutions converge asymptotically, that the limit is block-diagonal, and above all, that the limit matrix is defined uniquely as follows: For $n$ even, a block-diagonal matrix containing $2\times 2$ blocks, such that the super-diagonal entries are sorted by strictly increasing absolute value. Furthermore, the off-diagonal entries in these $2\times 2$ blocks have the same sign as the respective entries in the matrix employed as initial condition. For $n$ odd, there is one additional $1\times 1$ block containing a zero that is the top left entry of the limit matrix. The results presented here extend some early work of Kac and van Moerbeke.
\end{abstract}

	\maketitle
	
	\section{Introduction and Main Result}
	\label{s:intro}
		The tasks of sorting a list, diagonalizing a matrix, and solving a linear programming problem are traditionally solved with computer science algorithms, for example the quicksort algorithm for sorting or the simplex method for solving linear programs. Brockett \citep{ref:Brockett-91} showed that solutions to such problems can also be obtained by means of a smooth dynamical system, in particular as the limit of  solutions to certain matrix-valued ordinary differential equations (o.d.e.'s). A classical problem from linear algebra is therefore solvable by calculus. Motivated by Brockett's work, new problems, conventionally tackled by algebraic methods, have been assigned to calculus. For instance, \citep{ref:Faybusovich-92} proposed an ordinary differential equation (structurally similar to the one proposed by Brockett) as the starting point in a general approach to interior point methods for linear programming.


By a \emph{Jacobi matrix} we mean a symmetric tridiagonal matrix (in general, infinite) with real entries and distinct eigenvalues. 
In this article we present a matrix-valued ordinary differential equation which asymptotically block-diagonalizes a finite-dimensional zero-diagonal Jacobi matrix taken as its initial condition.  Jacobi matrices arise in a variety of applications, for example in solid state physics to characterize the Toda lattice, which is a simple model for a one dimensional crystal---see e.g. \citep{ref:Moser-75}, \citep[pp.\ 59-60]{helmke} for a detailed study. 
There is also a strong connection between Brockett's double bracket flow \citep{ref:Brockett-91} and the Toda lattice equation, which was first observed by \citep{Bloch-90}. 

		We offer a second motivation here that has intrinsic appeal and interest, relating to the computation of the roots of certain polynomials. Orthogonal polynomials on the real line corresponding to a Borel probability measure $\mu$ have considerable applications in mathematical physics and engineering \citep{Simon-05}. Let $\inprod{\cdot}{\cdot}$ denote the standard inner product on the Hilbert space $\Lp 2(\R, \mu)$. Then a sequence of monic orthogonal polynomials on the real line is defined recursively \citep{Szego} by 
		\begin{equation*}
				\begin{aligned}
					& P_{n+1}(x) = x P_n(x) - a_n^2 P_{n-1}(x) - b_{n+1} P_n(x),\quad n\in\N,\\
					& P_{-1}(\cdot) = 0,\quad P_0(\cdot) = 1,
				\end{aligned}
		\end{equation*} where
		\[
		\begin{aligned}
			a_n		& \Let \frac{\norm{P_{n}}}{\norm{P_{n-1}}},\text{ for }n\in\N \text{ and }
			b_{n+1}	\Let \frac{\inprod{xP_n}{P_n}}{\norm{P_{n}}^2}\text{ for } n\in\Nz.
		\end{aligned}
		\]
		To a given measure $\mu$, there corresponds a (generally infinite) Jacobi matrix
		\begin{equation} \label{e:jacobi:inf:dim} 
		\Jacobi \Let
		\begin{pmatrix}
			b_1		& a_1		& 0			& 0			& \cdots\\
			a_1		& b_2		& a_2		& 0			& \cdots\\
			0		& a_2		& b_3		& a_3		& \cdots\\
			0		& 0			& a_3		& b_4		& \cdots\\
			\vdots	& \vdots	& \vdots	& \vdots	& \ddots\\
		\end{pmatrix}
		\end{equation}
		with strict positive off-diagonal entries derived from its orthogonal polynomials, and it is well-known that the zeros of these orthogonal polynomials are precisely the eigenvalues of finite truncations of this Jacobi matrix. Conversely, Favard's Theorem \citep{ref:Favard-35} shows that  to every finite-dimensional symmetric tridiagonal matrix with strictly positive off-diagonal entries there corresponds a finitely supported measure. (The uniqueness of this measure for an infinite-dimensional Jacobi matrix is an issue that relates to the solvability of the ``moment problem'' \citep{ref:Akh-65}, with which we shall not deal here.) Since the space of square integrable functions corresponding to this measure is finite dimensional (the measure itself being finitely supported), it is enough that from the orthogonal polynomials it is possible to recover the measure, and this leads to the problem of finding the roots of such polynomials. The latter, in general, is known to be a difficult task. As mentioned above, 
for our applications it suffices to determine the eigenvalues of the finite truncations of the corresponding Jacobi matrix, which we shall do in this article with the aid of an appropriate matrix-valued ordinary differential equation. For this we further specialize the measures to non-negative linear combinations of finitely many Dirac measures on the real line placed symmetrically around $0$. These measures give rise to zero-diagonal Jacobi matrices, and our main result to the problem of finding the roots of the orthogonal polynomials corresponding to these measures.
		
		As a concrete application we consider the Gaussian quadrature method. For a given positive weight function $w(x)$ on an interval $[a,b]$ of the real line, the $n$-point Gaussian quadrature rule approximates an integral by
		\begin{equation*}
		 \sum_{i=1}^n w_i f(x_i) \approx \int_{a}^b f(x) w(x) \drv t.
		\end{equation*}
		It is known (see \cite[p.\ 17]{Simon-05}, \cite[p.\ 351]{Szego}, \cite[p.\ 21]{ref:watkins-05}) that the optimal points $x_1,\hdots,x_n$ coincide with the zeros of the orthogonal polynomial $P_n(x)$ introduced above, when we use the probability measure $\mu=\frac{w(x)\drv x}{\int_a^b w(x) \drv x}$, with $\drv x$ being the Lebesgue measure. As mentioned, these zeros are precisely the eigenvalues of a certain $n \times n$ truncation of the Jacobi matrix in \eqref{e:jacobi:inf:dim}. Furthermore, if the interval $[a,b]$ is symmetric about $0$ and the weight function $w(x)$ is even, the resulting Jacobi matrix has zero diagonal entries.

		In this article, we treat the problem of obtaining the eigenvalues of zero-diagonal finite dimensional Jacobi matrices from the asymptotic limit of a smooth dynamical system. Preparatory to stating our main result, we need some preliminary notation: We let $\sym(n)$ and $\skw(n)$ denote respectively the set of symmetric and skew-symmetric $n\times n$ matrices with real entries. We define $\jac_0(n)$ as the set of all $n\times n$ Jacobi matrices with real entries and zeros on its diagonal.
		For a matrix $A$, $\norm{A}$ is the Frobenius norm defined as $\norm{A} = \sqrt{\tr (A A^\top)}$. Let $\mathcal{C}(\R_{\geq 0},\R)$ denote the set continuous functions from $\R_{\geq 0}$ to $\R$. The bracket $[\cdot , \cdot ]$ is the usual matrix commutator $[A,B] = AB - BA$. For $H_0\in \sym(n)$ we let $\Ms(H_0) \Let \{\Theta\transp H_0 \Theta\mid \Theta\in \orth(n,\R)\}$ denote the set of all real matrices orthogonally similar to $H_0$ \citep{moorePaper}. For $H_0\in\jac_0(n)$ we define $\Mj(H_0)$ to be the set of all zero-diagonal Jacobi matrices that are isospectral  to $H_0$ (that is, they have the same eigenvalues). Moreover, for $p = 0, 1, \ldots, n-1$, we let
		\begin{align*}
		\D_{u,p}
		(a_1,a_2,\hdots,a_{n-p})
		&\Let 
				\begin{pmatrix}
					0		& \hdots		&0			& a_1			& 			& 		& \\
							& 0			& \hdots		&0			& a_2			& 		& \\
							& 			& \ddots		&			&			& \ddots	& \\
							&			&			& 			&			& 0		& a_{n-p}\\
							&			&			& 			& 			&		& 0\\
							&			&			& 			& 			&		& \vdots\\
							& 			&			& 			& 			&		& 0
				\end{pmatrix}\in\R^{n\times n}
				\end{align*}
				and
				\begin{align*}
		\D_{p}(a_1,a_2,\hdots,a_{n-p}) &\Let  \D_{u,p}(a_1,a_2,\hdots,a_{n-p}) +\D_{u,p} (a_1,a_2,\hdots,a_{n-p})\transp.
		\end{align*}

		The following is our main result:
		\begin{theorem} \label{t:mainTheoremKvM}
			Let $n$ be a positive integer. Consider the zero-diagonal Jacobi matrix
			\begin{equation}
				H =
					\begin{pmatrix}
						0 	& 	a_1 	& 		& 		&	 	& 		& 		&  	\\
						a_1 	& 	0 	& 	a_2 	& 		& 		& 		& 		& 	\\
							& 	a_2 	&  		& 		& 		& 		&		&	\\
							&     		&  		&	\ddots	& 		& 		& 		&	\\
							&  		&  		& 		& 		&	a_{n-2} & 		&	\\
							&  		&  		& 		&	a_{n-2} &	0 	&	a_{n-1} &	\\
							&  		&  		& 		&	0 	&	a_{n-1} &	0	&	\\
					\end{pmatrix} \in \jac_0(n),
			\end{equation}
			and the skew-symmetric matrix
			\begin{equation}
			\label{Kac_K(H)}
			\begin{aligned}
				K(H)\Let &\D_{u,2}(a_1a_2, a_2a_3,\cdots,a_{n-2}a_{n-1})\\
					& \qquad -\D_{u,2}(a_1a_2,a_2a_3,\cdots,a_{n-2}a_{n-1})\transp,
			\end{aligned}
			\end{equation}
			derived from $H$. Consider the matrix-valued o.d.e.\
			\begin{equation}
			\label{e:chTwo:KvMflow}
			\begin{aligned}
				\frac{\drv}{\drv t}H(t) & =[H(t),K(H(t))],
			\end{aligned}
			\qquad H(0) \Let H_0\in \jac_0(n).
			\end{equation} 
			\begin{enumerate}[label={\rm (\roman*)}, leftmargin=*, align=right, widest=iii]
				\item \label{cond:KvM:isospectral} \eqref{e:chTwo:KvMflow} defines an isospectral flow on the set of all Jacobi matrices with zero diagonal entries.
				\item \label{cond:KvM:eqpoints} The solutions $H(t)$ of \eqref{e:chTwo:KvMflow} exist for all $t\ge 0$ and approach  asymptotically the set of equilibrium points of \eqref{e:chTwo:KvMflow}, which  is given by 
					\[
						\bigl\{\bar{H}\in \Mj(H_0)\,\big|\, K(\bar{H})=0\bigr\}.
					\]
				\item \label{cond:KvM:convergence} $\lim_{t\to\infty}H(t)$ exists. 
				\item \label{cond:KvM:sorting}If $n$ is even, then
				  for all initial conditions $H_0\in\jac_0(n)$ other than the equilibria, where $H_0= \D_1(a_1,\hdots,a_{n-1})$ and $\spectrum(H_0)=\left \{ \pm \lambda_1,\hdots,\pm \lambda_{\frac{n}{2}}  \right \}$ with $|\lambda_1|<|\lambda_2|<\hdots <|\lambda_{\frac{n}{2}}|$, the solution $H(t)$ of \eqref{e:chTwo:KvMflow} converges to 
				  \begin{equation*}
				      \lim_{t\to\infty} H(t) = \D_1 \left(\sgn(a_1)|\lambda_1|, 0,\sgn(a_3)|\lambda_2|,0,\hdots,\sgn(a_{n-1})|\lambda_{\frac{n}{2}}|\right).
				  \end{equation*}
				   If $n$ is odd, then
				    for all initial conditions $H_0\in\jac_0(n)$ other than the equilibria, where $H_0= \D_1(a_1,\hdots,a_{n-1})$ and $\spectrum(H_0)=\left \{ 0,\pm \lambda_1,\hdots,\pm \lambda_{\frac{n-1}{2}}  \right \}$ with $0<|\lambda_1|<|\lambda_2|<\hdots <|\lambda_{\frac{n-1}{2}}|$, the solution $H(t)$ of \eqref{e:chTwo:KvMflow} converges to 
				    \begin{equation*}
				      \lim_{t\to\infty} H(t) = \D_1 \left(0,\sgn(a_2)|\lambda_1|, 0,\sgn(a_4)|\lambda_2|,0,\hdots,\sgn(a_{n-1})|\lambda_{\frac{n-1}{2}}|\right).
				  \end{equation*}
			\end{enumerate}
		\end{theorem}

\medskip

		\begin{remarksnn}\mbox{}
			\begin{enumerate}[label=\arabic*., leftmargin=*, align=left, widest=3]
				 \item The assertions $\ref{cond:KvM:isospectral}-\ref{cond:KvM:convergence}$ imply that the o.d.e.\ \eqref{e:chTwo:KvMflow} evolves on the set of zero-diagonal Jacobi matrices with a fixed spectrum determined by its initial condition. Solutions exist for all $t\geq 0$, attain a limit as $t\to \infty$ and approach the set of equilibrium points (having cardinality greater than $1$ for $n>2$) asymptotically. For all initial conditions $H_0\in \jac_0(n)$ other than the equilibria, however, property \ref{cond:KvM:sorting} defines the limit matrix uniquely as follows: For $n$ even, a block-diagonal matrix containing $2\times 2$ blocks, such that the super-diagonal entries are sorted by strictly increasing absolute value. Furthermore, the off-diagonal entries in these $2\times 2$ blocks have the same sign as the respective entries in the matrix employed as initial condition. For $n$ odd, there is one additional $1\times 1$ block containing a zero that is the top left entry of the limit matrix.
				\item \citep{ref:Brockett-91} studied the o.d.e.\ \eqref{e:chTwo:KvMflow} in which the map $K$ on the right-hand side of \eqref{e:chTwo:KvMflow} was set to $K(H) = [H, N]$, where $N$ is a constant symmetric matrix. In our case, however, $K(H) = [H, N(H)]$, where \(N\) is a linear function of $H$, not a constant. As an important consequence of the definition of $K$ as in \eqref{Kac_K(H)}, all equilibrium points of \eqref{e:chTwo:KvMflow} are non-hyperbolic. In contrast, for $K(H) = [H, N]$, where $N$ is a constant symmetric matrix, all the equilibrium points are known to be hyperbolic \citep{ref:Brockett-91}. In particular, the proof techniques in \citep{ref:Brockett-91} do not carry over, and in order to prove the sorting property \eqref{cond:KvM:sorting} in Theorem \ref{t:mainTheoremKvM}---the stable manifold theorems \citep[p.\ 362ff]{helmke} cannot be employed. As such, the analysis of \eqref{e:chTwo:KvMflow} requires new tools.
				\item It is possible to treat $N$ as a parameter in \citep{ref:Brockett-91}, and can therefore be employed, e.g., to sort the eigenvalues of $\lim_{t\to\infty} H(t)$ according to a particular order by selecting an appropriate matrix $N$. Such variations in our case are not readily available since \(K\) is a fixed function defined by \eqref{Kac_K(H)}.
				\item The dynamical system $\frac{\drv}{\drv t}H(t)=[K(H(t)),H(t)]$ was studied first in \citep{kacVanMoerbeke}. This article studied several properties of \eqref{e:chTwo:KvMflow} by considering the dynamics of the individual components, and the techniques relied on properties of the orthogonal polynomials associated to Jacobi matrices. In contrast, our technical tools are system theoretic. The analysis of the properties of \eqref{e:chTwo:KvMflow} from a double bracket perspective, to our knowledge, has been carried out here for the first time. In addition, the sorting property \eqref{cond:KvM:sorting} in Theorem \ref{t:mainTheoremKvM} is an entirely new observation.
			\end{enumerate}
		\end{remarksnn}

		\textit{Outline of the article:} \secref{s:mainresult} presents the proof of the Theorem, which we illustrate with some numerical examples in \secref{s:examples}. We conclude in \secref{s:concl} with a summary of our work and comment on possible subjects of further research.
	
  \section{Proof of Theorem \ref{t:mainTheoremKvM}}
	\label{s:mainresult}
		Some preliminaries are needed in order to prove Theorem \ref{t:mainTheoremKvM}. We begin with the following classical result, which will play a key role behind proving that the solutions to \eqref{e:chTwo:KvMflow} are isospectral. 
		\begin{proposition}[\citep{ref:Lax-68}]
		\label{prop:ch1:isospectralflows}
		Let $\varphi:$ $\sym(n)\lra $ $\skw(n)$ be a smooth mapping, and suppose that 
		$\gamma: \R _{\geq 0}\lra \sym(n)$ is a curve satisfying
		\begin{equation} \label{eq: bracket}
		\frac{\drv\gamma}{\drv t}(t)=[\varphi \circ \gamma (t),\gamma (t)],\qquad t\ge 0.
		\end{equation}
		Then there exists a smooth family of unitary matrices $(U(t))_{t\geq 0}$ with $U(0)=I_{n}$ such that
		\[
			\gamma (t)=U(t)^{-1}\gamma (0) U(t),\qquad t \geq 0.
		\]
		The family $(\gamma (t))_{t\geq 0}$ is thus isospectral.
		\end{proposition}

		Next we define the mapping
		\begin{equation}\label{e:define:N}
		\begin{aligned}
		  \sym(n) & \ni  A \Let
				\begin{pmatrix}
					a_{11}		& a_{12}	& \hdots 	& 		& \\
					a_{12}		& a_{22}	& 		& 		& \\
					\vdots		& 		& \ddots	& 		& \\
							&		&		& 		& \\
							&		&		& a_{n-1,n-1}	& a_{n-1,n}\\
							& 		&		& a_{n-1,n}	& a_{nn}
				\end{pmatrix}\mapsto\\
				& \qquad \qquad N(A)\Let\D_{1}(-a_{12},0,a_{34},2a_{45},\hdots,(n-3)a_{n-1,n})\in\sym(n).
		\end{aligned}
		\end{equation}
		The mapping $N(\cdot)$ is linear, and can be written as
		\begin{equation} \label{e:repN}
		N(A)=\sum_{i=1}^{n-1}(i-2)\bigl(E_iAE_{i+1}+E_{i+1}AE_i\bigr),
		\end{equation}
		where $A\in\sym(n)$ and $E_i$ is the matrix with $1$ at its $(i,i)$-th entry and zeros elsewhere.
		\begin{proposition} 
		\label{p:ch2:kvmDbf}
		Let 
		\begin{equation}
		H_n= 
		\begin{pmatrix}
			0 	& 	a_1 	& 	0	& 		&	 	& 		& 		&  	\\
			a_1 	& 	0 	& 	a_2 	& 		& 		& 		& 		& 	\\
			0	& 	a_2 	&  		& 		& 		& 		&		&	\\
				&     		&  		&	\ddots	& 		& 		& 		&	\\
				&  		&  		& 		& 		&	a_{n-2} & 	0	&	\\
				&  		&  		& 		&	a_{n-2} &	0 	&	a_{n-1} &	\\
				&  		&  		& 		&	0 	&	a_{n-1} &	0	&	\\
					\end{pmatrix}
		\in \jac_0(n).
		\end{equation}
		Then the commutator of $H_n$ and $N(H_n)$ is given by
		\begin{equation} \label{e:chTwo:K}
		\begin{aligned}
		K(H_n) & \Let [H_n,N(H_n)]\\
		&= \D_{u,2}(a_1a_2,\hdots,a_{n-2}a_{n-1})-\D_{u,2}(a_1a_2,\hdots,a_{n-2}a_{n-1})\transp,
		\end{aligned}
		\end{equation}
		where $N(\cdot)$ is the linear mapping defined in \eqref{e:define:N}.
		\end{proposition}
		\begin{proof}
		The proof proceeds by induction. Observe that $\D_{u, 2} = 0$ if $n = 1$. For $n=1$, we have $H_n=0$, and therefore $K(H_n)=0$. For $n=2$, we have $H_n=\left(\begin{smallmatrix}
		                                                                                             0 & a_1 \\ a_1 & 0
		                                                                                            \end{smallmatrix} \right)$
		and $N(H_n)=\left(\begin{smallmatrix}
		                        0 & -a_1 \\ -a_1 & 0
		                  \end{smallmatrix} \right)$.
		Since $H_n$ and $N(H_n)$ commute, $K(H_n)=0$.
		
		We consider the induction step:
		\begin{equation}
		H_{n+1}=\begin{pmatrix}
			  & & & & 0     &        0\\
			  & & \boxed{H_{n-1}}& &\vdots &   \vdots\\
			  & & & & 0     &        0\\
			  & & & &a_{n-1}     &        0\\
			0&\hdots& 0&a_{n-1}&0 & a_n \\
			0&\hdots& 0& 0 &a_n & 0 
			\end{pmatrix}
		\end{equation}
		and
		\begin{equation}
		N(H_{n+1})=\begin{pmatrix}
			  & & & & 0     &        0\\
			  & & \boxed{N(H_{n-1})}& &\vdots &   \vdots\\
			  & & & & 0     &        0\\
			  & & & &(n-3)a_{n-1}     &        0\\
			0&\hdots& 0&(n-3)a_{n-1}&0 & (n-2)a_n \\
			0&\hdots& 0& 0 &(n-2)a_n & 0 
			\end{pmatrix}.
		\end{equation}
		Simple matrix multiplications lead to
		\begin{equation} \label{HN(H_n)+1}
		\begin{aligned}
		H_{n+1}N(H_{n+1}) &=
		\begin{pmatrix}
			  & & & & & 0     &        0\\
			  & & & \boxed{H_{n-1}N(H_{n-1})} & &\vdots &   \vdots\\
			  & &  & & & 0     &        0\\
			  & & &  & & (n-3)a_{n-1}a_{n-2}     &        0\\ 
			  & & & & & 0    &        \kreis{3}\\   
			0&\hdots& 0&(n-4)a_{n-1}a_{n-2}&0& \kreis{4} & 0 \\
			0&\hdots& 0& 0 & \kreis{1} &0 & \kreis{2} 
			\end{pmatrix},\\
		\end{aligned}
		\end{equation}
		where
		\begin{align*}
		\kreis{1} &= (n-3)a_na_{n-1}, \\
		\kreis{2} &= (n-2)a^2_n, \\
		\kreis{3} &= (n-2)a_na_{n-1}, \\
		\kreis{4} &= (n-3)a^2_{n-1}+(n-2)a^2_n,
		\end{align*}
		and
		\begin{equation} \label{NH_n+1}
		\begin{aligned}
		N(H_{n+1})H_{n+1}&=
		\begin{pmatrix}
			  & & & & & 0     &        0\\
			  & & & \boxed{N(H_{n-1})H_{n-1}} & &\vdots &   \vdots\\
			  & &  & & & 0     &        0\\
			  & & &  & & (n-4)a_{n-1}a_{n-2}     &        0\\ 
			  & & & & & 0    &        \kreis{6}\\   
			0&\hdots& 0&(n-3)a_{n-1}a_{n-2}&0& \kreis{4} & 0 \\
			0&\hdots& 0& 0 & \kreis{5} &0 & \kreis{2} 
			\end{pmatrix}, \\
		\end{aligned}
		\end{equation}
		where
		\begin{align*}
		\kreis{5} = (n-2)a_na_{n-1}, \\
		\kreis{6} = (n-3)a_na_{n-1}.
		\end{align*}
		By the induction hypothesis,
		\[
			[H_{n-1},N(H_{n-1})]=H_{n-1}N(H_{n-1})-N(H_{n-1})H_{n-1}=K(H_{n-1}).
		\]
		Therefore, by \eqref{HN(H_n)+1} and \eqref{NH_n+1} we get
		\begin{align*}
		[H_{n+1},N(H_{n+1})]&=
		\begin{pmatrix}
			  & & & & & 0     &        0\\
			  & & & \boxed{K(H_{n-1})} & &\vdots &   \vdots\\
			  & &  & & & 0     &        0\\
			  & & &  & & a_{n-1}a_{n-2}     &        0\\ 
			  & & & & & 0    &        a_na_{n-1}\\   
			0&\hdots& 0&-a_{n-1}a_{n-2}&0& 0 & 0 \\
			0&\hdots& 0& 0 & -a_na_{n-1} &0 & 0 
			\end{pmatrix}\\
		&=K(H_{n+1}). \qedhere
		\end{align*}
		\end{proof}
		Note, that in view of Proposition \ref{p:ch2:kvmDbf}, the modified Kac-van Moerbeke equation \eqref{e:chTwo:KvMflow} can be represented as the double bracket o.d.e.
		\begin{equation}
		\label{e:KvM_DBF}
				\frac{\drv}{\drv t}H(t)=\left[H(t),[H(t),N(H(t))]\right],
			\qquad H_0\Let H(0)\in \jac_0(n).
			\end{equation}
		We shall employ the following auxiliary lemma in the proof of Theorem \ref{t:mainTheoremKvM}.
		\begin{lemma}
		\label{lemma:chTwo:stayJacobi}
		If $A=\D_1(a_1,a_2,\hdots,a_{n-1})$ with $a_i\in \R$ for all $i=1,\hdots,n-1$, then 
		\begin{equation}
		 \label{e:lemma:RHS}
		\begin{aligned}
			& [A,[A,N(A)]] =\\
			& \qquad \D_{1}(-a_1a_2^2,-a_2a_3^2+a_1^2a_2,\hdots ,-a_{n-2}a_{n-1}^2+a_{n-3}^2a_{n-2}, a_{n-2}^2a_{n-1}),
		\end{aligned}
		\end{equation}
		where the mapping $N(\cdot)$ is defined in \eqref{e:define:N}.
		\end{lemma}
		\begin{proof}
		We have already shown in Proposition \ref{p:ch2:kvmDbf} that $K(A) = [A, N(A)]$. Let $A_1$ denote the upper triangular part and $A_2$ the lower triangular part of $A$, such that $A=A_1+A_2$. Using an analogous decomposition for $K(A)$ we 
		get $K_1(A) + K_2(A)=K(A)$. We abbreviate and simply write $K_i$ for $K_i(A)$, $i\in\{1,2\}$. We decompose
		\begin{equation} \label{iso_rhs}
		\begin{aligned}
		[A,K(A)]&=[A_1+A_2,K_1+K_2]\\
		&=[A_1,K_1]+[A_1,K_2]+[A_2,K_1]+[A_2,K_2].
		\end{aligned}
		\end{equation}
		We first show by induction that $[A_1,K_1]=0$. Indeed, if we denote with a superscript the size of a matrix and consider $A^n=\D_1(a_1,a_2,\hdots,a_{n-1})$, i.e., $A_1^n=\D_{u,1}(a_1,a_2,\hdots,a_{n-1})$ and $A_2^n=\D_{u,1}(a_1,a_2,\hdots,a_{n-1})\transp$, we observe that for $n=1$ we have $[A_1^n,K_1^n]=0$. The induction step can be done as follows:
		\begin{align*}
		  [A_1^{n+1}, &\  K_1^{n+1}]\\
		  &=A_1^{n+1}K_1^{n+1}-K_1^{n+1}A_1^{n+1}\\
		&= 
		\begin{pmatrix}
		    &&&0\\
		    &\boxed{A_1^{n}}&&\vdots\\
		    &&&0\\
		    &&&a_n \\ 
		  0&\hdots&0&0
		  \end{pmatrix}
		\begin{pmatrix}
		    &&&0\\
		    &\boxed{K_1^{n}}&&\vdots\\
		    &&&a_{n-1}a_n\\
		    &&&0 \\ 
		  0&\hdots&0&0
		  \end{pmatrix}\\
		&\quad -
		\begin{pmatrix}
		    &&&0\\
		    &\boxed{K_1^{n}}&&\vdots\\
		    &&&a_{n-1}a_n\\
		    &&&0 \\ 
		  0&\hdots&0&0
		  \end{pmatrix}
		\begin{pmatrix}
		    &&&0\\
		    &\boxed{A_1^{n}}&&\vdots\\
		    &&&0\\
		    &&&a_n \\ 
		  0&\hdots&0&0
		  \end{pmatrix}\\
		&=
		\begin{pmatrix}
		    &&&0\\
		    &\boxed{A_1^nK_1^{n}}&&\vdots\\
		    &&&a_{n-2}a_{n-1}a_n\\
		    &&&0 \\
		    &&&0 \\ 
		  0&\hdots&0&0
		  \end{pmatrix}-
		\begin{pmatrix}
		    &&&0\\
		    &\boxed{K_1^nA_1^{n}}&&\vdots\\
		    &&&a_{n-2}a_{n-1}a_n\\
		    &&&0 \\
		    &&&0 \\ 
		  0&\hdots&0&0
		  \end{pmatrix}\\
		&=0.
		\end{align*}
		Moreover, we have $A\transp_1=A_2$ and $K\transp_1=-K_2$. Therefore, 
		\[
		[A_2,K_2]=A_2K_2-K_2A_2=-A\transp_1K\transp_1+K\transp_1A\transp_1=(A_1K_1-K_1A_1)\transp=[A_1,K_1]\transp.
		\]
		Thus, $[A_2,K_2]=0$ as well, and it remains to show (again by induction) that 
		\begin{align*}
			& [A_2,K_1]+[A_1,K_2]\\
			& \quad = \D_{1}(-a_1a_2^2,-a_2a_3^2+a_1^2a_2,\hdots ,-a_{n-2}a_{n-1}^2+a_{n-3}^2a_{n-2}, a_{n-2}^2a_{n-1}).
		\end{align*}
		As the next step, we claim that
		\[
		[A_2^n,K_1^n]= \D_{u,1}(-a_1a_2^2,-a_2a_3^2+a_1^2a_2,\hdots ,-a_{n-2}a_{n-1}^2+a_{n-3}^2a_{n-2}, a_{n-2}^2a_{n-1}).
		\]
		The induction base is trivial. Then we have:
		\begin{align*}
		[A_2^{n+1}, &\ K_1^{n+1}]\\
		& =A_2^{n+1}K_1^{n+1}-K_1^{n+1}A_2^{n+1}\\
		&=   
		\begin{pmatrix}
		    &&&0\\
		    &\boxed{A_2^{n}}&&\vdots\\
		    &&&0\\
		    &&&0 \\ 
		  0&\hdots&a_n&0
		  \end{pmatrix}
		  \begin{pmatrix}
		    &&&0\\
		    &\boxed{K_1^{n}}&&\vdots\\
		    &&&a_{n-1}a_n\\
		    &&&0 \\ 
		  0&\hdots&0&0
		  \end{pmatrix}\\
		&\quad -
		\begin{pmatrix}
		    &&&0\\
		    &\boxed{K_1^{n}}&&\vdots\\
		    &&&a_{n-1}a_n\\
		    &&&0 \\ 
		  0&\hdots&0&0
		  \end{pmatrix}
		\begin{pmatrix}
		    &&&0\\
		    &\boxed{A_2^{n}}&&\vdots\\
		    &&&0\\
		    &&&0 \\ 
		  0&\hdots&a_n&0
		  \end{pmatrix}\\
		&= 
		  \begin{pmatrix}
		    &&&0\\
		    &\boxed{A_2^{n}K_1^n}&&\vdots\\
		    &&&0\\
		    &&0&a_{n-1}^2a_n \\ 
		  0&\hdots&0&0
		  \end{pmatrix}\\
		&\quad -
		\begin{pmatrix}
		    &&&0\\
		    &\boxed{K_1^{n}A_2^n}&&\vdots\\
		    &&&0\\
		    &&&0 \\ 
		  0&\hdots&0&0
		  \end{pmatrix}-
		  \begin{pmatrix}
		    &&&0\\
		    &0&&\vdots\\
		    &&a_{n-1}a_n^2&0\\
		    &&0&0 \\ 
		  0&\hdots&0&0
		  \end{pmatrix} \\
		&=\D_{u,1}(-a_1a_2^2,-a_2a_3^2+a_1^2a_2,\hdots ,-a_{n-1}a_{n}^2+a_{n-2}^2a_{n-1}, a_{n-1}^2a_{n}).
		\end{align*}
		Again, in view of $A\transp_1=A_2$ and $K\transp_1=-K_2$, we obtain
		\[
		[A_1,K_2]=A_1K_2-K_2A_1=-A\transp_2K\transp_1+K\transp_1A\transp_2=(A_2K_1-K_1A_2)\transp=[A_2,K_1]\transp.
		\]
		Thus, 
		\[
		[A_1^n,K_2^n]=\D_{u,1}(a_1a_2^2,a_2a_3^2+a_1^2a_2,\hdots ,-a_{n-2}a_{n-1}^2+a_{n-3}^2a_{n-2}, a_{n-2}^2a_{n-1})\transp. 
		\]
		According to \eqref{iso_rhs}, for $A=\D_1(a_1,a_2,\hdots,a_{n-1})$ we have
		\[
		[A,K(A)]=\D_{1}(-a_1a_2^2,-a_2a_3^2+a_1^2a_2,\hdots ,-a_{n-2}a_{n-1}^2+a_{n-3}^2a_{n-2}, a_{n-2}^2a_{n-1}),
		\]
		which completes the proof.
		\end{proof}
	      \begin{lemma}\label{lemma:chTwo:charEquilibira}
	      For $A, B\in \sym(n)$,
	      \[
	      [A,[A,B]]=0 \text{ if and only if } [A,B]=0.
	      \]
	      \end{lemma}
	      \begin{proof}
	      The ``if'' part is trivial. To prove the ``only if'' part, suppose that $[A,[A,B]]=0$. This implies $B[A,[A,B]]=0$. Using the techniques in 
	      \citep[p.\ 49]{helmke}, we compute
	      \begin{align*}
	      0&=\tr\left(B[A,[A,B]] \right) \\
		&=\tr\left(B(A^2B-2ABA+BA^2)\right) \\
		&=\tr\left(BA^2B-2BABA+B^2A^2\right) \\
		&=\tr\left(BA^2B-BABA-ABAB+AB^2A  \right) \\
		&=\tr\left((BA-AB)(AB-BA)\right) \\
		&=\tr\left([B,A][A,B]\right) \\
		&=\tr\bigl([A,B]\transp [A,B]\bigr) \\
		&=\norm{[A,B]}^2,
	      \end{align*}
	      which immediately gives $[A,B]=0$.\qedhere
	      \end{proof}	
	      \begin{lemma} \label{lemma:chTwo:symFunction}
	    Consider the continuous function
	    \begin{equation}\label{e:ch:sym:functionfromlemma}
	    M(H_0)\ni H \mapsto f(H)\Let -\frac{1}{4} \|H-N(H)\|^2+\frac{1}{4} \|N(H)\|^2\in\R.
	    \end{equation}
	    With respect to the o.d.e.\ 
	    \begin{equation}
		\label{e:dblb}
	    \frac{\drv}{\drv t}H(t)=\left[H(t),\left[H(t),N(H(t))\right]\right], \quad H(0)=H_0\in \sym(n),
	    \end{equation}
	    the time derivative of $f(H(\cdot))$ is given by
	    \begin{equation} \label{e:ch:Two:timederivativefunction}
	    \frac{\drv}{\drv t}f(H(t))=\|[H(t),N(H(t))]\|^2.
	    \end{equation}
	    \end{lemma}

	    \begin{proof}
	    We start by simplifying the function
	    \begin{align*}
	    f(H(t)) &=-\frac{1}{4} \|H(t)-N(H(t))\|^2+\frac{1}{4} \|N(H(t))\|^2 \\
		    &=-\frac{1}{4} \tr \left(H(t)H(t)-H(t)N(H(t))-N(H(t))H(t)+N(H(t))N(H(t)) \right)\\
			    & \qquad +\frac{1}{4} \|N(H(t))\|^2 \\
		    &=-\frac{1}{4} \|H(t)\|^2+\frac{1}{2}\tr \left(N(H(t))H(t) \right)-\frac{1}{4} \|N(H(t))\|^2+\frac{1}{4} \|N(H(t))\|^2 \\
		    &=-\frac{1}{4} \|H(t)\|^2+\frac{1}{2}\tr \left(N(H(t))H(t) \right).
	    \end{align*}
	    Since $H(t)\in \Ms(H_0)$ by Proposition \ref{prop:ch1:isospectralflows}, $\|H(t)\|$ is constant for all $t\geq 0 $. We calculate the derivative of $f$ along the trajectories of \eqref{e:dblb} as follows:
	    \begin{align}
	    \frac{\drv}{\drv t}f(H(t))
	    &=\frac{1}{2}\tr \left( \left(\frac{\drv}{\drv t}N(H(t)) \right)H(t) +N(H(t))\dot{H}(t)\right)\nn \\
	    &=\frac{1}{2}\tr \left( \left(\frac{\drv}{\drv t}N(H(t)) \right)H(t) \right)\nn\\
	    & \qquad +\frac{1}{2}\tr \left(N(H(t)) \left[H(t),[H(t),N(H(t))]\right] \right)\nn \\
	    &=\frac{1}{2}\tr \left( \left(\frac{\drv}{\drv t}N(H(t)) \right)H(t) \right)\nn\\
	    & \qquad + \frac{1}{2}\tr \left(\underbrace{[N(H(t)),H(t)]}_{[H(t),N(H(t))]\transp}[H(t),N(H(t))] \right)\nn \\
	    &=\frac{1}{2}\tr \left( \left(\frac{\drv}{\drv t}N(H(t)) \right)H(t) \right) + \frac{1}{2}\|[H(t),N(H(t))]\|^2,\label{e:ch:sym:derivfsecond}
	    \end{align}
	    where, at the third equality, we employed the fact \citep[p.\ 162]{ref:Ber-09} that for $A,B,C\in \R^{n\times n}$,  $\tr (A[B,C])=\tr([A,B]C)$. Therefore, it remains to show that 
	    \[
		    \tr \left( \left(\frac{\drv}{\drv t}N(H(t)) \right)H(t) \right)=\|[H(t),N(H(t))]\|^2.
	    \]
	    Note that
	    \[
	    \frac{\drv}{\drv t}N(H(t))=N\left(\dot{H}(t)\right)
	    \]
	    since $N$ is linear, and since from \eqref{e:ch:sym:derivfsecond} it follows that $\tr(N(H(t))\dot{H}(t))=\|[H(t),N(H(t))]\|^2$, our proof will be complete if we show that 
		\begin{equation}
		\label{e:equality of traces}
			\tr\bigl(N(\dot{H}(t))H(t)\bigr)=\tr\bigl(N(H(t))\dot{H}(t)\bigr).
		\end{equation}
	    To this end, employing the expansion of $N$ in \eqref{e:repN}, we see that
	    \begin{align*}
	    \tr\left( N(  \dot{H}(t)  )H(t)\right) &=\tr \left( \sum_{i=1}^{n-1}(i-2)\left(E_i\dot{H}(t)E_{i+1}+E_{i+1}\dot{H}(t)E_i\right)H(t) \right) \\
	    &=\sum_{i=1}^{n-1}(i-2) \tr\left( \left( E_i\dot{H}(t)E_{i+1}+E_{i+1}\dot{H}(t)E_i\right)H(t) \right) \\
	    &= \sum_{i=1}^{n-1}(i-2) \left( \tr\left( E_i\dot{H}(t)E_{i+1}H(t) \right)+\tr \left(E_{i+1}\dot{H}(t)E_iH(t)\right) \right) \\
	    &=\sum_{i=1}^{n-1}(i-2) \left( \tr\left( E_{i+1}H(t)E_i\dot{H}(t) \right)+\tr \left(E_iH(t)E_{i+1}\dot{H}(t)\right) \right) \\
	    &=\sum_{i=1}^{n-1}(i-2) \tr\left( \left( E_iH(t)E_{i+1}+E_{i+1}H(t)E_i\right)\dot{H}(t) \right) \\
	    &=\tr\left( N(H(t))\dot{H}(t)\right),
	    \end{align*}
	    which establishes \eqref{e:equality of traces}, and completes the proof.
	    \end{proof}
	    \begin{lemma}\label{lemma:chTwo:numberEquilibriaPoints}
	    The o.d.e.\ 
	    \begin{equation} \label{e:chTwo:lemmaEquilibriapoints}
	    \frac{\drv}{\drv t}H(t)=\left[H(t),\left[H(t),N(H(t))\right]\right], \quad H(0)=H_0\in \jac_0(n),
	    \end{equation}
	    has a finite number of equilibrium points on $\jac_0(n)$ that are isospectral to $H_0$.
	    \end{lemma}

	    \begin{proof}
	    In view of Proposition \ref{prop:ch1:isospectralflows} and Lemma \ref{lemma:chTwo:stayJacobi}, $H(t)\in\Mj(H_0)$ for all solutions of \eqref{e:chTwo:lemmaEquilibriapoints} and for all $t\geq 0$. Moreover, for $H_0\in\jac_0(n)$ we have $\Mj(H_0)\subseteq \jac_0(n).$
	    By Lemma \ref{lemma:chTwo:charEquilibira}, 
	    \begin{equation} \label{e:chTwo:def:E}
	    \begin{aligned}                                              
	    E&\Let\bigl\{ H(t)\in \Mj(H_0)\,\big|\,\norm{[H(t),N(H(t))]}^2=0 \bigr\} \\
	    &\text{ }=\bigl\{ H(t)\in \Mj(H_0) \,\big|\, [H(t),N(H(t))]=0 \bigr\}
	    \end{aligned}
	    \end{equation}
	    is the set of all equilibrium points of \eqref{e:chTwo:lemmaEquilibriapoints} on $\jac_0(n)$ that are isospectral to $H_0$. Let
		    \[
			    \tilde{H} =
			    \D_1(a_1,a_2,a_3,a_4,\hdots,a_{n-2},a_{n-1})\in \jac_0(n).
		    \]
	    At this point it is crucial to recall that according to our definition Jacobi matrices have distinct eigenvalues.
	    We treat the case of $n$ even and $n$ odd separately:

	    \noindent\boxed{$n$ \text{ even}}\\
	    Consider the set of matrices
		\[
			\tilde{E} \Let \bigl\{\tilde{H}\in\jac_0(n) \,\big|\, [\tilde{H},N(\tilde{H})]=0 \bigr\}.
		\]
	    In view of \eqref{e:chTwo:K}, the only possibility for $\tilde{H}$ to lie in $\tilde{E}$ is if $a_i=0$ for all $i$ even and $a_i\neq 0,$ for all $i\in\{1,3,\hdots n-1\}$ with $a_i\neq a_j$ for all $i\neq j$, such that 
	    \begin{equation*}
		    \tilde{H}_{\tilde{E}}= \D_1(a_1,0,a_3,0,\hdots,0,a_{n-1}).
	    \end{equation*} 
	    Note that $\tilde{H}_{\tilde{E}}$ has the spectrum $\spectrum(\tilde{H}_{\tilde{E}})=\left\{\pm a_i\,\big|\, i\in\{1,3,\hdots n-1\}\right\}$ containing only distinct eigenvalues. Now $E$ as defined in \eqref{e:chTwo:def:E} is a subset of $\tilde{E}$ satisfying the isospectral conditions; it is the restriction of $\tilde{E}$ to the set of zero-diagonal Jacobi matrices isospectral to $H_0$, i.e., $E=\tilde{E}|_{\Mj(H_0)}$. 
	    Considering all the possible permutations of the $a_i$ for $i\in\{1,3,\hdots, n-1\}$, the set $E$ contains $\bigl(\frac{n}{2}\bigr)!$ equilibrium points on $\jac_0(n)$ that are isospectral to $H_0$. 

	    \noindent\boxed{$n$ \text{ odd}}\\
	    First of all, since $\tilde{H}=\D_1(a_1,a_2,a_3,a_4,\hdots,a_{n-2},a_{n-1})$ is a (zero-diagonal) Jacobi matrix, we need to evoke the fact \citep{compactManifold} that its spectrum has the form
	    \begin{equation} \label{e:chTwo:eigenvalues:Penskoi}
	     \spectrum(\tilde{H})=\{0,\pm \lambda_1, \pm \lambda_2, \hdots, \pm \lambda_{\frac{n-1}{2}}\}, \quad \text{where }\lambda_i\neq \lambda_j \text{ for all } i\neq j.
	    \end{equation}
	     Moreover, if we consider 
		\[
			\tilde{E}\Let \bigl\{\tilde{H}\in\jac_0(n) \,\big|\, [\tilde{H},N(\tilde{H})]=0 \bigr\},
		\]
	      in view of \eqref{e:chTwo:K}, $\tilde{H}$ must satisfy
	     \begin{equation} \label{e:chTwo:secondCond}
	      a_ia_{i+1}=0 \quad \text{for all }i=1,\hdots, n-2
	     \end{equation}
	      in order to lie in $\tilde{E}$. Furthermore, in view of \eqref{e:chTwo:eigenvalues:Penskoi} and \eqref{e:chTwo:secondCond} $\tilde{H}$ has to be a block-diagonal matrix containing $\frac{n-1}{2}$ blocks of the form
	      $\left(\begin{smallmatrix}
	       0 & a_{j} \\a_{j} & 0
	      \end{smallmatrix}\right)$ with $a_j\neq0$ and one $1\times 1$ block containing a zero, where the block entries are distinct (since the eigenvalues of of $\tilde{H}$ need to be distinct). Accordingly, there are $\frac{n+1}{2}$ possibilities to place the $1\times 1$ block in $\tilde{H}$. As above $E=\tilde{E}|_{\Mj(H_0)}$ and considering all possible permutations of the blocks we get that $E$ contains $\bigl( \frac{n+1}{2} \bigr)\bigl(\frac{n-1}{2}\bigr)!$ equilibrium points on $\jac_0(n)$ that are isospectral to $H_0$.
 \end{proof} 	
 
	    \begin{lemma} \label{lemma:sorting}
		Let $g\in\mathcal{C}(\R_{\geq 0},\R)$ and suppose that $\lim_{t\to\infty}g(t)$ exists. If
		\begin{equation*}
		 \lim_{t\to\infty} |g(t)|=\eta, \text{ for some }\eta>0 
		\end{equation*}
		and
		\begin{equation*}
		 \lim_{t\to\infty} \int_0^t g(s) \drv s = -\infty,
		\end{equation*}
		then
		\begin{equation*}
		 \lim_{t\to\infty} g(t) < 0.
		\end{equation*}
	    \end{lemma}
	    \begin{proof}
	       First, note that $\lim_{t\to\infty}|g(t)|=\eta>0$ implies that there exists $T_0>0$ such that for all $t>T_0$, $g(t)\neq 0$. Therefore, by continuity, the sign of $g(t)$ is the same for all $t>T_0$. Suppose $\lim_{t\to\infty} g(t)\geq 0$. If the limit is equal to $0$, then its absolute value would have to converge to $0$ as well, which is a contradiction. Therefore, $\lim_{t\to\infty}g(t)>0.$ But then $g(t)=|g(t)|$ for all $t>T_0$ and $\lim_{t\to\infty}g(t)=\eta.$ 
	       Then, there exists $T_1>0$ such that for all $t>T_1$, $g(t)>\frac{\eta}{2}$. Without loss of generality we assume that $T_1>T_0$ (otherwise put $T_1'\Let\max\{T_0,T_1\}$). By continuity, $g$ is bounded on the interval $[0,T_1]$ by some constant $C$. Therefore,
	       \begin{equation*}
	        -\infty < -T_1C \leq \int_0^{T_1}g(s)\drv s \leq T_1 C < \infty.
	       \end{equation*}
		Now we see that for $t>T_1$
		\begin{equation*}
		 \int_0^t g(s) \drv s = \int_0^{T_1}g(s) \drv s + \int_{T_1}^t g(s) \drv s \geq -T_1 C + \int_{T_1}^t \frac{\eta}{2} \drv s \to \infty -T_1C = \infty
		\end{equation*}
		for $t\to\infty$. Therefore, $\lim_{t\to\infty} \int_0^t g(s) \drv s = \infty$ which is a contradiction.
	      \end{proof}

\begin{proof}[Proof of Theorem \ref{t:mainTheoremKvM}] 
	In view of Lemma \ref{lemma:chTwo:stayJacobi} it follows that the right-hand side of \eqref{e:chTwo:KvMflow} is a symmetric tridiagonal matrix with zero diagonal entries given by \eqref{e:lemma:RHS}. The fact that \eqref{e:chTwo:KvMflow} is isospectral is an immediate consequence of Proposition \ref{prop:ch1:isospectralflows}. Therefore, the flow of \eqref{e:chTwo:KvMflow} evolves on the set of zero-diagonal Jacobi matrices isospectral to $H_0$, i.e., $H(t)\in \Mj(H_0) $ for all $t\ge 0$. This settles the claim in \ref{cond:KvM:isospectral}.

	In order to show \ref{cond:KvM:eqpoints}, note that in view of \ref{cond:KvM:isospectral}, $H(t)\in \Mj(H_0)$ for all $t\ge 0$. Since $\Mj(H_0)$ is known to be a compact manifold \citep[Proposition 1.2]{compactManifold}, $H(t)$ exists for all $t\ge 0$. By Lemma \ref{lemma:chTwo:charEquilibira} we see that the set of equilibrium points $\bar{H}$ of \eqref{e:chTwo:KvMflow} is given by $\bigl\{\bar{H}\in\Mj(H_0)\,\big|\,[\bar{H},N(\bar{H})]=0\bigr \}$. To show that $H(t)$ approaches the set of equilibrium points, consider the function
	\begin{equation}\label{e:ch:Two:function}
	\begin{aligned}
		\Mj(H_0) \ni H \mapsto f(H) \Let -\frac{1}{4} \|H(t)-N(H(t))\|^2+\frac{1}{4} \|N(H(t))\|^2 \in \R.
	\end{aligned}
	\end{equation}
	According to Lemma \ref{lemma:chTwo:symFunction},
	\begin{equation} 
	 \frac{\drv}{\drv t}f(H(t))=\|[H(t),N(H(t))]\|^2\geq 0.
	\end{equation}
	We invoke the Krasovskij-LaSalle's Invariance Principle \citep[Theorem 4.4]{khalil}, \citep[p.\ 178]{vidyasagar}: First, we define the set
	\begin{align*}
	E  \Let&  \left\{ H \in \Mj(H_0)\,\Big|\,[H,N(H)]=0 \right\}.
	\end{align*}
	By Lemma \ref{lemma:chTwo:charEquilibira}, $E$ coincides with the set of all equilibrium points of \eqref{e:chTwo:KvMflow}. Therefore, $E$ is an invariant set with respect to \eqref{e:chTwo:KvMflow}. Second, recall that $\Mj(H_0)$ is a compact set. Therefore, by the Krasovskij-LaSalle's Invariance Principle, every solution $(H(t))_{t\geq 0}$ starting in $\Mj(H_0)$ approaches the set of equilibrium points $E$ asymptotically, which proves the claim in \ref{cond:KvM:eqpoints}.

	To show property \ref{cond:KvM:convergence} note that we have already shown in part \ref{cond:KvM:eqpoints} that every solution of \eqref{e:chTwo:KvMflow} approaches the set of equilibrium points asymptotically. However, according to Lemma \ref{lemma:chTwo:numberEquilibriaPoints}, the number of equilibrium points is finite. By continuity of trajectories, therefore, $(H(t))_{t\ge 0}$ converges to a single equilibrium point, i.e., $\lim_{t\to\infty}H(t)$ exists, which proves the claim in \ref{cond:KvM:convergence}.
	
	To prove \ref{cond:KvM:sorting}, consider first the o.d.e.\
	\begin{equation} \label{eq:proof:sorting}
	 \frac{\drv x(t)}{\drv t} = x(t) g(t), \qquad x(0)=x_0,
	\end{equation}
	where \(x(t)\in\R\), $g\in\mathcal{C}(\R_{\geq 0},\R)$. Suppose that $x_0\neq 0$. The unique solution (\citep{khalil}) to \eqref{eq:proof:sorting} is given by 
	\begin{equation*}
	 x(t) = x_0 \exp\left( \int_0^t g(s) \drv s \right),\quad t\ge 0,
	\end{equation*}
	and it follows at once that 
	\begin{equation}
	\label{eq:proof:sorting:asymptote}
		\lim_{t\to\infty}x(t) = 0 \quad\text{implies}\quad \lim_{t\to\infty} \int_0^t g(s) \drv s = -\infty.
	\end{equation}
	We consider the components of the o.d.e.\ \eqref{e:chTwo:KvMflow}:
	\begin{align*}
	 \frac{\drv a_1(t)}{\drv t} &= -a_1(t) a_2^2(t),\\
	 \frac{\drv a_i(t)}{\drv t} &= a_i(t) \left(a_{i-1}^2(t)-a_{i+1}^2(t)\right),\quad i = 2, \ldots, n-2,\\
	 \frac{\drv a_{n-1}(t)}{\drv t} &= a_{n-1}(t)a_{n-2}^2(t),
	\end{align*}
	and distinguish two different cases depending on the parity of $n$:

	\noindent\boxed{$n$ \text{ even}}\\
	For each \(i = 1, \ldots, \frac{n}{2}-1\), define the function $g_{2i}\in \mathcal{C}(\R_{\geq 0},\R)$ by $g_{2i}(t) \Let a_{2i-1}^2(t)-a_{2i+1}^2(t)$. With the definition of \(g_{2i}\), the even-numbered components of the o.d.e.\ \eqref{e:chTwo:KvMflow} can be represented as
	\begin{equation*}
	 \frac{\drv a_{2i}(t)}{\drv t}=a_{2i}(t)\left(a_{2i-1}^2(t)-a_{2i+1}^2(t)  \right) =a_{2i}(t) g_{2i}(t). 
	\end{equation*}
	This o.d.e.\ is of the form \eqref{eq:proof:sorting}, and by Lemma \ref{lemma:chTwo:numberEquilibriaPoints} and properties \eqref{cond:KvM:eqpoints} and \eqref{cond:KvM:convergence} of Theorem \ref{t:mainTheoremKvM}, it follows that $\lim_{t\to\infty}a_{2i}(t)$ exists and is equal to zero. By \eqref{eq:proof:sorting:asymptote}, 
	\begin{equation}\label{e:proof:sorting:n_even_1}
	 \lim_{t\to \infty}\int_0^t g_{2i}(s) \drv s = -\infty.
	\end{equation}
	Independently of the preceding steps, recall that the eigenvalues of any Jacobi matrix are distinct. From the definition of $g_{2i}$ it now follows that
	\begin{equation}\label{e:proof:sorting:n_even_2}
	 \text{there exists \(\eta_{2i} > 0\) such that} \quad \lim_{t\to \infty} |g_{2i}(t)|  = \eta_{2i}.
	\end{equation}
	In view of \eqref{e:proof:sorting:n_even_1} and \eqref{e:proof:sorting:n_even_2}, Lemma \ref{lemma:sorting} and property \ref{cond:KvM:convergence} of Theorem \ref{t:mainTheoremKvM} lead to
	\begin{equation}
	 \lim_{t\to\infty}a_{2i-1}^2(t) < \lim_{t\to\infty}a_{2i+1}^2(t)\quad \text{for all }i=1,\hdots,\tfrac{n}{2}-1.
	\end{equation}
	
	\noindent\boxed{$n$ \text{ odd}}\\
	We introduce the function $f(t) \Let a_{n-1}^2(t)$, \(t\ge 0\), where \(a_{n-1}(\cdot)\) is the solution of the final component of \eqref{e:chTwo:KvMflow}. The derivative of \(f\) given by $\frac{\drv f(t)}{\drv t} = 2a_{n-1}^2(t) a_{n-2}^2(t) \ge 0$ shows that $f$ is monotonically non-decreasing. Since \(a_{n-1}(0) \neq 0\), it follows from the fact that \(f\) is monotonically non-decreasing, that $\lim_{t\to\infty} a_{n-1}(t)\neq0$. In view of Lemma \ref{lemma:chTwo:numberEquilibriaPoints}, and properties \eqref{cond:KvM:eqpoints} and \eqref{cond:KvM:convergence} of Theorem \ref{t:mainTheoremKvM}, it follows that $\lim_{t\to\infty}a_{2i-1}(t)$ exists and is equal to zero for all $i=1,\hdots \frac{n-1}{2}$. As in the case of \(n\) even, we define, for each \(i = 2, \ldots, \frac{n-1}{2}\), a function $g_{2i-1}\in \mathcal{C}(\R_{\geq 0},\R)$ by $g_{2i-1}(t)\Let a_{2i-2}^2(t)-a_{2i}^2(t)$. With this definition of \(g_{2i-1}\), we see that the odd-
numbered components (greater than one) of the o.d.e.\
 \eqref{e:chTwo:KvMflow} can be represented as
	\begin{equation*}
	 \frac{\drv a_{2i-1}(t)}{\drv t}=a_{2i-1}(t)\left(a_{2i-2}^2(t)-a_{2i}^2(t)  \right) =a_{2i-1}(t) g_{2i-1}(t). 
	\end{equation*}
	This o.d.e.\ is of the form \eqref{eq:proof:sorting}, and since we know that $\lim_{t\to\infty}a_{2i-1}(t)=0$, it follows from \eqref{eq:proof:sorting:asymptote} that
	\begin{equation}\label{e:proof:sorting:n_even_1b}
	 \lim_{t\to \infty}\int_0^t g_{2i-1}(s) \drv s = -\infty.
	\end{equation}
	Since the eigenvalues of any Jacobi matrix are distinct, by definition of $g_{2i-1}$ we see that
	\begin{equation}\label{e:proof:sorting:n_even_2b}
	 \text{there exists \(\eta_{2i-1} > 0\) such that}\quad \lim_{t\to \infty} |g_{2i-1}(t)|  = \eta_{2i-1}.
	\end{equation}
	In view of \eqref{e:proof:sorting:n_even_1b} and \eqref{e:proof:sorting:n_even_2b}, Lemma \ref{lemma:sorting} and property \ref{cond:KvM:convergence} of Theorem \ref{t:mainTheoremKvM} lead to
	\begin{equation}
	 \lim_{t\to\infty}a_{2i-2}^2(t) < \lim_{t\to\infty}a_{2i}^2(t), \quad \text{for all }i=2,\hdots,\frac{n-1}{2}.
	\end{equation}

	Finally, note that independently of the parity of $n$, for each component of the o.d.e.\ \eqref{e:chTwo:KvMflow}, zero is always an equilibrium point. Therefore, if \(a_i(0)\neq 0\), then \(a_i(\cdot)\) cannot take the value \(0\), and this in turn implies that if $\lim_{t\to\infty}a_i(t)\neq 0$, then $\sgn(a_i(0))=\sgn(\lim_{t\to\infty}a_i(t))$ for all $i=1,\hdots,n-1$. This completes the proof. 
\end{proof}

	\section{Examples}
	\label{s:examples}
	We illustrate Theorem \ref{t:mainTheoremKvM} with three numerical examples. We solve the o.d.e.\ \eqref{e:chTwo:KvMflow} numerically and plot the upper-diagonal entries of $H(t)$ as functions of time for a chosen initial condition. The figures show that the solution of \eqref{e:chTwo:KvMflow} converges rather quickly to an equilibrium.
	
	\begin{example} \label{ex:chTwo:example4}
	We start with a $4\times 4$ example. Consider an initial condition 
	\begin{equation} \label{eq:chTwo:ex4jac_ic}
	H(0)=H_0=
	\D_1(5,-6,-2)
	\end{equation}
	with the spectrum 
	\begin{align*}
		\spectrum(H_0)=\bigl\{  & \pm 1.26,\pm 7.96\bigr\}.
	\end{align*}
	By solving the o.d.e.\ \eqref{e:chTwo:KvMflow} numerically with this initial condition we see again that the transient behavior vanishes well before 1 second of simulation. At $T=1s$ we have
	\begin{align*}
	H(T)&= 
	\D_1(1.26,0,-7.96).
	\end{align*}
	Figure \ref{Fig:chTwo:ex4_jac} shows the evolution of the super-diagonal components of $H(t)$ against time $t$.
	\begin{figure}[!htb]
	    \centering
	    \includegraphics[width=10cm]{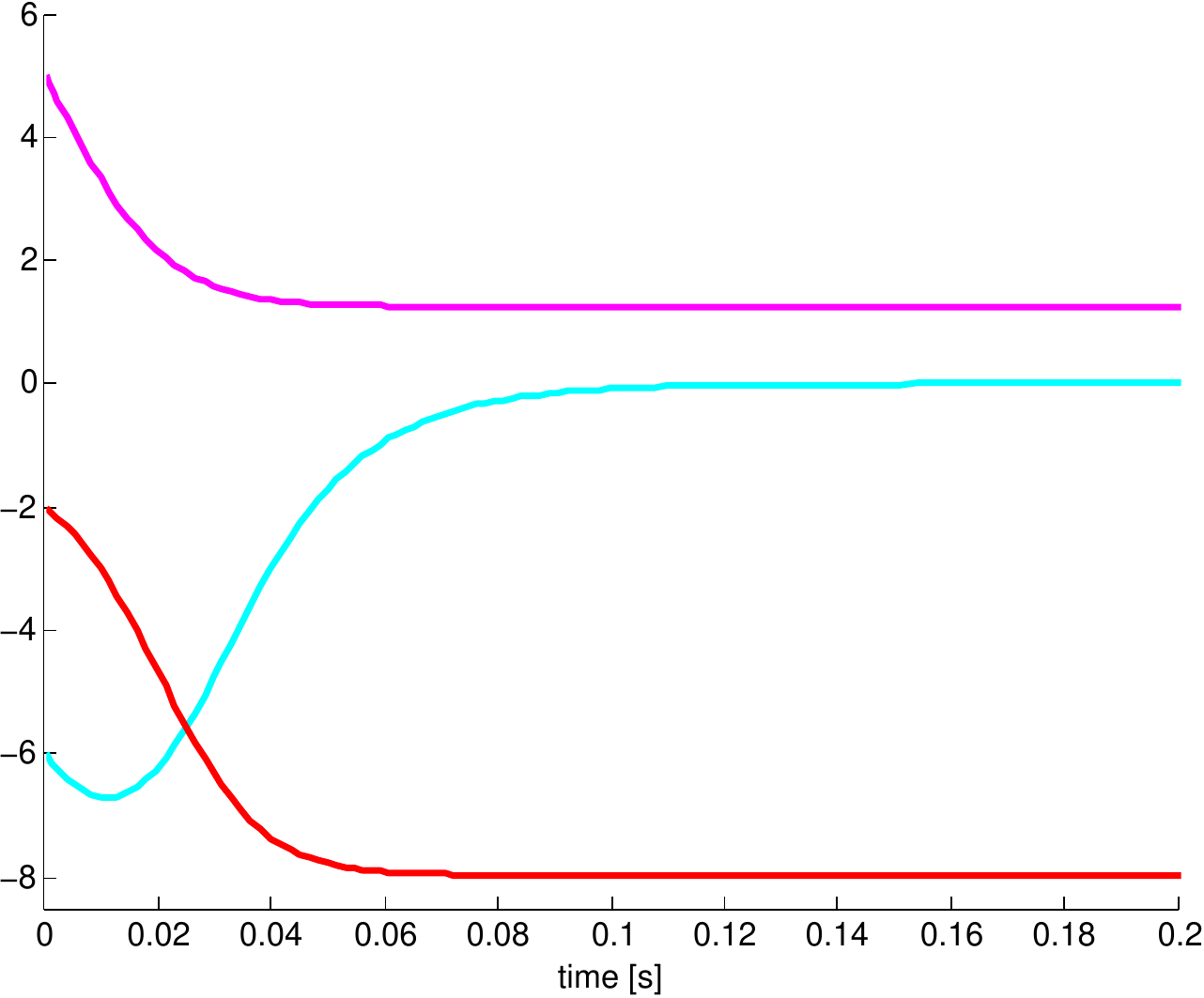}
	    \caption{Flow of o.d.e.\ \eqref{e:chTwo:KvMflow} for the initial condition \eqref{eq:chTwo:ex4jac_ic}.}
	    \label{Fig:chTwo:ex4_jac}
	\end{figure}
	\end{example}
      \begin{example} \label{ex:chTwo:example10}
	Consider the following initial condition of size \(10 \times 10\) 
	\begin{equation} \label{eq:ch2:ex10_ic}
	H(0)=H_0=
	\D_1(-3,10,1,-2,-6,-11,5,6,12),
	\end{equation}
	having the spectrum
	$\spectrum(H_0)=\left \{  \pm 0.21, \pm 2.71, \pm 10.48, \pm 12.34, \pm 14.36 \right\}$. By solving the o.d.e.\ \eqref{e:chTwo:KvMflow} numerically, we see that the transient behavior vanishes well before 1 second of simulation. At $T=1s$ we have
	\[
	H(T)= 
	\D_1(-0.21,0,2.71,0,-10.48,0,12.34,0,14.36).
	\]
	Figure \ref{Fig:ch2:ex10} shows the evolution of the super-diagonal components of $H(t)$ against time $t$.
	\begin{figure}[!htb]
	    \centering
	    \includegraphics[width=10cm]{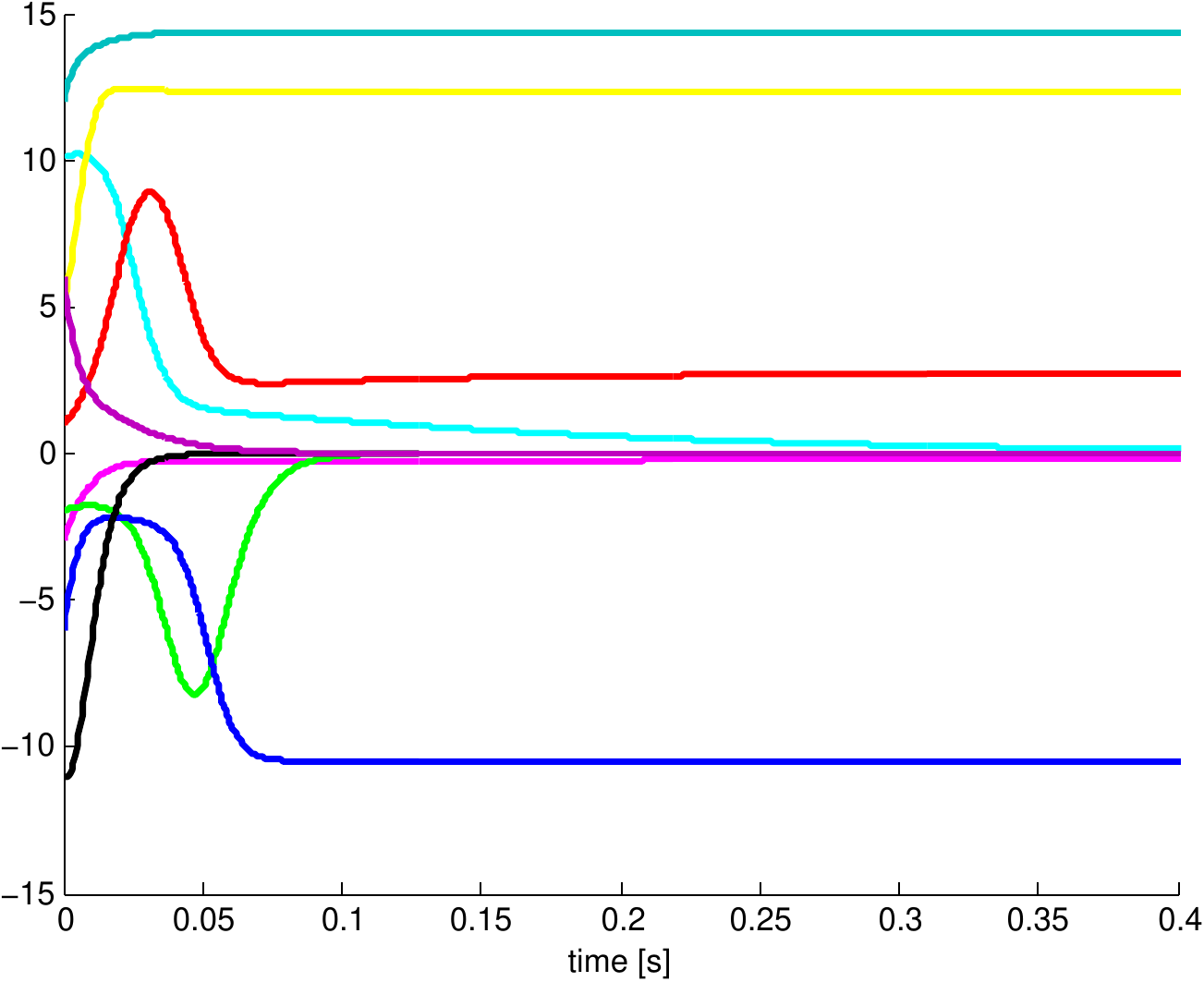}
	    \caption{Flow of o.d.e.\ \eqref{e:chTwo:KvMflow} for the initial condition \eqref{eq:ch2:ex10_ic}.}
	    \label{Fig:ch2:ex10}
	\end{figure}
	\end{example}
	\begin{example} \label{ex:chTwo:example29}
	Consider the following initial condition of size \(29 \times 29\) 
	\begin{equation} \label{eq:ch2:ex29_ic}
	\begin{aligned}
	H(0)=H_0&=
	\D_1(-6,7,-8,2,-13,7,-12,7,-2,9,2,-4,2,\\
	&\hspace{12mm} 4,6,-15,-7,11,-7, 9,9,15,1,5,-3,11,-1,-3),
	\end{aligned}
	\end{equation}
	having the spectrum
	$\spectrum(H_0)=\left \{ 0, \pm 2.81, \pm 2.98, \pm 4.17, \pm 4.66, \pm 4.84, \pm 6.26, \pm 9.29,\right.$\\ 
	$\left.\pm 10.84, \pm 11.53, \pm 11.83, \pm 12.48, \pm 17.11, \pm 17.98, \pm 18.85 \right\}$. By solving the o.d.e.\ \eqref{e:chTwo:KvMflow} numerically, we see that the transient behavior vanishes well before 1 second of simulation. At $T=1s$ we have
	\begin{align*}
	H(T)&= \D_1(0, 2.81, 0, 2.98, 0, 4.17, 0, 4.66, 0, 4.84, 0, -6.26, 0, 9.29, 0, -10.84, 0,\\ 
	&\hspace{16mm}11.53, 0, 11.83, 0, 12.48, 0, 17.11, 0, 17.98, 0, -18.85).
	\end{align*}
	Figure \ref{Fig:ch2:ex10} shows the evolution of the super-diagonal components of $H(t)$ against time $t$.
	\begin{figure}[!htb]
	    \centering
	    \includegraphics[width=10cm]{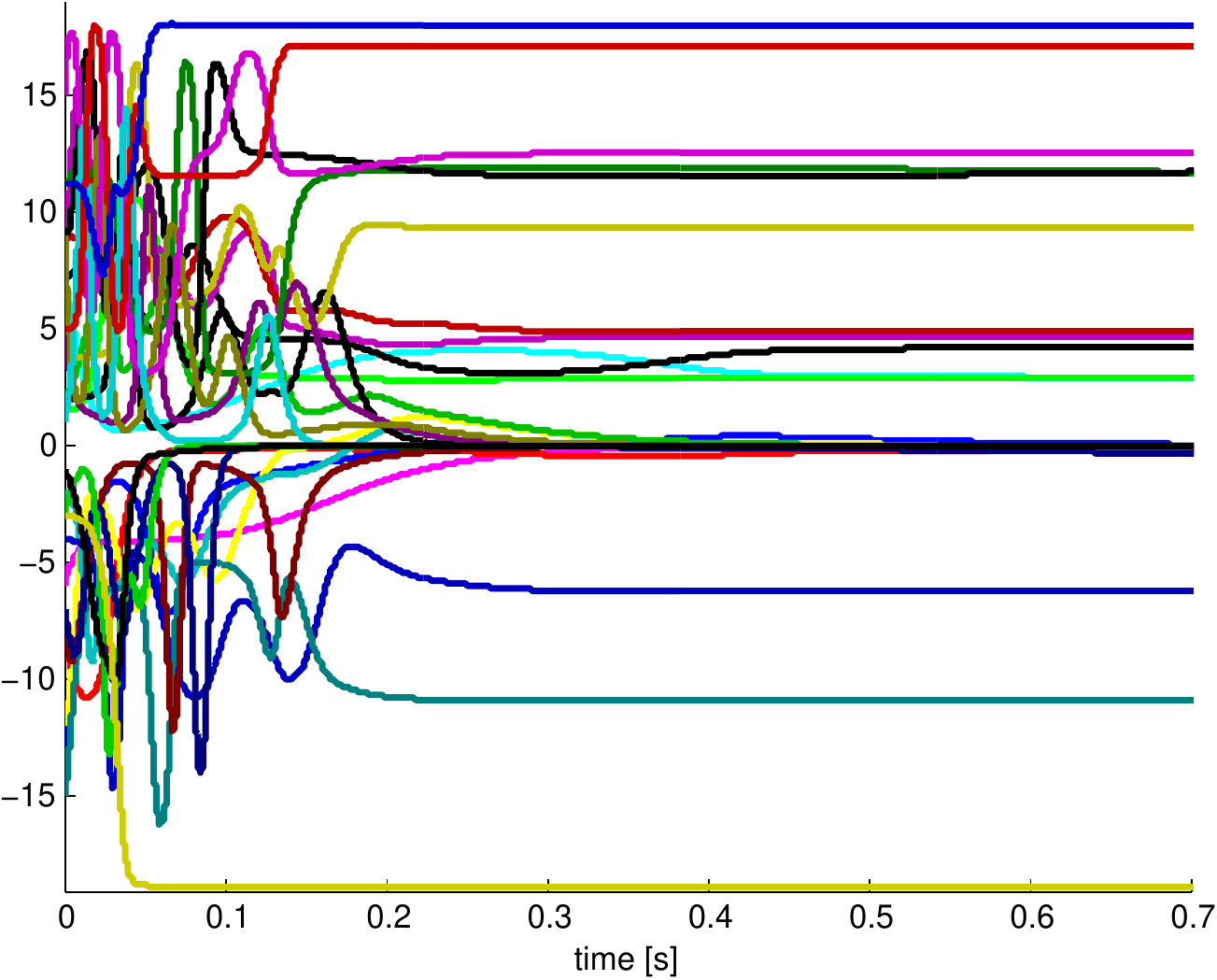}
	    \caption{Flow of o.d.e.\ \eqref{e:chTwo:KvMflow} for the initial condition \eqref{eq:ch2:ex29_ic}.}
	    \label{Fig:ch2:ex29}
	\end{figure}
	\end{example}

	\section{Conclusions and Future Direction}
	\label{s:concl}
	We presented a matrix-valued isospectral ordinary differential equation that asymptotically block-diagonalizes a finite-dimensional zero-diagonal Jacobi matrix 
	employed as its initial condition. We demonstrated that this o.d.e.\ can be represented as a double bracket equation, thus establishing a connection to \citep{ref:Brockett-91}, and we have proved certain new key properties of this o.d.e.\ by system-theoretic techniques. In particular that the limit is block-diagonal and the blocks of the limit point are sorted by increasing magnitude of the corresponding eigenvalue. 

		The domain of the o.d.e.\ \eqref{e:chTwo:KvMflow} can be expanded to the set of real symmetric matrices $\sym(n)$. Since $\Ms(H_0)$ for $H_0\in\sym(n)$ is again a compact manifold \citep{helmke}, assertions \ref{cond:KvM:isospectral} and \ref{cond:KvM:eqpoints} of Theorem \ref{t:mainTheoremKvM} hold also for the symmetric case, and the proof proceeds analogously. Extensive simulations lead us to conjecture that the solutions converge asymptotically to block diagonal matrices, as in the case of zero-diagonal Jacobi matrices employed as initial conditions. However, a proof for this conjecture is still an open problem; the primary technical difficulty arises from the fact that in contrast to the case of zero-diagonal Jacobi matrices, in this case there exist infinitely many equilibrium points.


		\bibliographystyle{agsm}
		\bibliography{ref}

\bigskip

\end{document}